\documentclass[twocolumn]{autart}    

\usepackage{comment}
\usepackage{graphicx}
\usepackage{mathptmx}
\usepackage{amsmath}
\usepackage{amssymb}
\usepackage{subfig}
\usepackage[round]{natbib}   

\begin{document}

\begin{frontmatter}
\runtitle{Safe Leader-Follower Density Control via HO-MF-CBFs}  

\title{Safe Leader-Follower Density Control via Higher-Order Mean-Field Control Barrier Functions}

\thanks[footnoteinfo]{This paper was not presented at any IFAC 
meeting. Corresponding author M. di Bernardo. Email: mario.dibernardo@unina.it}

\author[SSM]{Cinzia Tomaselli}\ead{c.tomaselli@ssmeridionale.it},     
\author[SSM,Unina]{Mario di Bernardo}\ead{mario.dibernardo@unina.it}          

\address[SSM]{Modeling and engineering risk and complexity, Scuola Superiore Meridionale, Naples, Italy}                                              
\address[Unina]{Department of Electrical Engineering and Information Technology,
University of Naples Federico II, Naples, Italy}            
          

\begin{abstract}                        
Mean-field control barrier functions (MF-CBFs) enforce safety of large-scale multi-agent systems at the level of densities, but existing formulations require the control to act directly on the population to be kept safe. We consider populations that are not directly actuated and are influenced only through other, controlled populations, possibly in the presence of diffusion. We show that if the constrained population lies at directed graph distance $d$ from the actuated one, then, under a nondegeneracy condition, the safety functional has relative degree $d+1$ and the control enters its $(d+1)$-th time derivative affinely. Higher-order safety constraints can therefore still be enforced through a convex quadratic program. We apply the framework to leader-follower density control among dangerous regions, where the follower safety functional has relative degree two and the leader safety functional has relative degree one. We characterize exactly when the two constraints are jointly infeasible and bound the tracking errors induced by the safety correction. Simulations in cluttered environments show that the proposed filter preserves follower safety, whereas enforcing leader safety alone does not.
\end{abstract}

\end{frontmatter}

\section{Introduction}
Leader-follower control coordinates large populations by acting on a small set of leaders, whose dynamics indirectly shape the collective behavior of the followers \citep{cruz1978leader,tanner2004leader,cao2012distributed,notarstefano2011containment}. It arises whenever only a fraction of the agents can be directly controlled, as in traffic regulation by a few connected autonomous vehicles \citep{stern2018dissipation,siri2021freeway}, swarm-based search and rescue \citep{queralta2020collaborative,dorigo2021swarm}, and robotic containment of pollutants \citep{kakalis2008robotic}. A paradigmatic example is shepherding \citep{strombom2014solving,lama2024shepherding,lama2025nonreciprocal}, where a few herders steer a larger population of targets toward a goal region, recently also in cluttered environments, where indirect steering and safety preservation are particularly challenging~\citep{tomaselli_icra2026,punzo2025decentralized}.

Continuum formulations recast these problems in the space of probability distributions~\citep{maffettone2025leader-follower,salzano2026robust}, mitigating the curse of dimensionality and enabling rigorous guarantees on convergence, robustness and performance. However, they focus on convergence to a desired follower distribution and do not address safety constraints at the macroscopic level. This is challenging because the followers are not directly actuated: their safety must be guaranteed by acting only on the leader density.

Control barrier functions (CBFs) represent a powerful tool to enforce safety constraints in real time for nonlinear control systems \citep{ames2016control,ames2019control}. By imposing suitable differential inequalities on a safety function, CBFs provide constraints that make a prescribed safe set forward invariant. This framework has been extended in several directions, including stochastic CBFs (S-CBFs) \citep{clark2021control,wang2021chance,wang2021safety}, which allow safety filters to be imposed on stochastic dynamics, exponential CBFs~\citep{nguyen2016exponential} and high-order CBFs (HO-CBFs) \citep{xiao2021high,breeden2023robust}, which address safety constraints with relative degree larger than one, and adaptive CBFs (A-CBFs) \citep{xiao2021adaptive,parwana2025rate}, which improve feasibility by tuning barrier parameters online when safety requirements become restrictive or potentially conflicting. Higher-order barrier certificates have also been designed for leader-follower networks of finitely many agents, where follower constraints are affected by the leader inputs only through the network interconnections~\citep{sharifi2023higher}.
Related developments have also addressed multi-agent safety through non-smooth barrier functions, enabling Boolean compositions of multiple system-level constraints in multi-robot systems~\citep{glotfelter2017nonsmooth}.

At the population level, mean-field control barrier functions (MF-CBFs) \citep{fung2025mean} have recently been introduced to enforce safety constraints directly on density dynamics. These methods extend the classical CBF framework to mean-field models and provide safety filters for large-scale multi-agent systems. Recent extensions include Banach-space control barrier functions \citep{gao2026banach} which generalize control barrier functions to Banach spaces, and mean-field control barrier functions for stochastic multi-agent systems \citep{tomaselli2026mean}, whose mean-field evolution is described by advection-diffusion equations. In the latter setting, to solve a shepherding problem, leader-follower interactions were considered under the assumption that the follower velocity field could be directly adjusted by the safety filter. However, in many leader-follower systems, followers are not directly actuated: their motion is shaped only indirectly through the controlled leader population. The first-order MF-CBF formulations in \citep{fung2025mean,tomaselli2026mean} are therefore insufficient for this setting, since they require the control input to appear explicitly in the first time derivative of the safety functional.

Concurrently with and independently of the present work, higher-order MF-CBFs have been proposed in \citep{nurbekyan2026higher}. That work defines the relative degree of a mean-field safety functional with respect to control-affine individual-agent dynamics $\dot{\mathbf x}= \mathbf f+g\mathbf u$, provides a general calculus for its higher-order derivatives along the corresponding continuity equation, and applies it to cross- and self-correlation barriers for double-integrator swarms, where position-only constraints have relative degree two. In \citep{nurbekyan2026higher}, the constrained density and the actuated density coincide: relative degree two arises because the control acts on the velocity component of the agents' state, and the dynamics is deterministic. Here, we address a structurally different mechanism, in which the agents to be kept safe belong to a \emph{different, non-actuated} population, whose density obeys advection-diffusion dynamics and is affected by the control only through a nonlocal interaction field generated by the actuated population. This cross-population, indirect actuation is not covered by the single-population framework of \citep{nurbekyan2026higher}, and it raises issues that do not appear there, namely the coexistence of constraints of different relative degree on the leader and follower populations and their possible mutual conflict, as well as the effect of the safety correction on the follower tracking performance. It also differs from the agent-level leader-follower setting of~\citep{sharifi2023higher}: here, barriers are defined on densities, and the relative degree is determined by the directed interaction graph among populations, independently of the number of agents.

Building on the recursive barrier condition of finite-dimensional HO-CBFs \citep{xiao2021high}, recalled here for densities consistently with \citep{nurbekyan2026higher}, the contributions of this paper are as follows.
First, we characterize the relative degree of safety functionals in interconnected mean-field systems through the directed interaction graph: if the constrained population is at graph distance $d$ from the actuated one, the control first enters the $(d+1)$-th time derivative of the safety functional, and does so affinely, so that the higher-order MF-CBF condition is a single affine constraint on the control field.
Second, we specialize this result to leader-follower density control in cluttered environments, where the follower and leader safety functionals have relative degree two and one, respectively, and are enforced through a single convex quadratic program acting on the leader control field.
Third, we give a necessary and sufficient condition for the joint feasibility of the leader and follower safety constraints.
Finally, we bound how the safety correction propagates to the leader and follower tracking errors, and show numerically that enforcing leader safety alone through a first-order MF-CBF does not preserve follower safety, whereas the proposed higher-order filter does.

\section{Mathematical preliminaries}
\label{sec:math_preliminaries}
Throughout the paper, $\nabla$, $\nabla \cdot$, and $\Delta$ denote the spatial gradient, divergence, and Laplacian operators, respectively. For a time-dependent quantity $f$, $\partial_t f$ denotes its partial derivative with respect to time, whereas $\dot f$ denotes its total time derivative. For a vector field $\mathbf f$, $J_{\mathbf x}\mathbf f$ denotes its Jacobian matrix with respect to the variable $\mathbf x$, namely $(J_{\mathbf x}\mathbf f)_{ij}=\partial f_i/\partial x_j$.

The spatial domain $\Omega=[-\Lambda,\Lambda]^2$ is endowed with periodic boundary conditions, i.e.\ it is identified with the flat torus. The norms $\|\cdot\|_1$, $\|\cdot\|_2$, and $\|\cdot\|_\infty$ denote the norms in $\mathcal L^1(\Omega)$, $\mathcal L^2(\Omega)$, and $\mathcal L^\infty(\Omega)$, respectively, and $\|f\|_{W^{1,\infty}}:=\|f\|_\infty+\|\nabla f\|_\infty$. We denote by $\langle f,g\rangle:=\int_\Omega f(\mathbf x)g(\mathbf x)\,\mathrm d\mathbf x$ the $\mathcal L^2(\Omega)$ inner product, with the Euclidean scalar product in the integrand for vector-valued functions, and by $f*g$ the periodic convolution on $\Omega$.

Given a functional $\mathcal F(\rho)$, its first variation (functional derivative) $\partial^f_\rho\mathcal F(\rho):\Omega\to\mathbb R$ is the function such that
\[
\frac{\mathrm d}{\mathrm d\varepsilon}\mathcal F(\rho+\varepsilon\eta)\Big|_{\varepsilon=0}
=
\int_\Omega \partial^f_\rho\mathcal F(\rho)(\mathbf x)\,\eta(\mathbf x)\,\mathrm d\mathbf x
\]
for all admissible perturbations $\eta$ with $\int_\Omega\eta\,\mathrm d\mathbf x=0$, i.e.\ the $\mathcal L^2$ representative of the Fr\'echet derivative of $\mathcal F$~\citep{ilin1982fundamentals}.

A continuous function $\alpha:\mathbb R_{\ge0}\to\mathbb R_{\ge0}$ is of class-$\mathcal K$ if it is strictly increasing and $\alpha(0)=0$; a continuous function $\alpha:\mathbb R\to\mathbb R$ is of extended class-$\mathcal K$ if it is strictly increasing and $\alpha(0)=0$, so that $\alpha(s)<0$ for $s<0$. All such functions are assumed to be locally Lipschitz, so that the comparison lemma~\citep{khalil2002nonlinear} can be invoked.

\emph{Control barrier functions.}
For a system $\dot{\mathbf x}=\mathbf w(t,\mathbf x,\mathbf u)$ and a continuously differentiable function $h$, a control barrier function (CBF) renders the safe set $\mathcal C=\{\mathbf x: h(\mathbf x)\ge0\}$ forward invariant by enforcing $\dot h+\alpha(h)\ge0$ along trajectories, with $\alpha$ of extended class-$\mathcal K$ and $h(\mathbf x(0))\ge0$~\citep{ames2016control}.

At the population level, for a density obeying $\partial_t \rho+ \nabla \cdot (\rho\, \mathbf{w}(t,\mathbf{x},\mathbf{u}))= 0$ and a safe set $\mathcal{C} = \{ \rho : \mathcal H(\rho) \geq 0 \}$, the first-order MF-CBF condition~\citep{fung2025mean} is
$\dot{\mathcal H}(\rho) + \alpha(\mathcal H) \geq 0$, where integration by parts, under boundary conditions ensuring that the boundary term vanishes, yields
\[
\dot{\mathcal H} = \int_\Omega \partial^f_\rho \mathcal{H}(\rho)\, \partial_t \rho\, \mathrm d \mathbf x = \int_\Omega \rho\,\nabla\partial^f_\rho\mathcal{H}(\rho)\cdot\mathbf{w}(t,\mathbf{x},\mathbf{u})\, \mathrm d \mathbf{x}.
\]
In particular, if $\mathcal{H}(\rho(0,\cdot))\geq 0$ and the MF-CBF condition is enforced along trajectories, then $\mathcal{C}$ is forward invariant.

\section{Mathematical modeling and problem statement}
\label{sec:model}
Consider a large swarm composed of two populations of agents, referred to as leaders and followers, evolving in the periodic domain $\Omega=[-\Lambda,\Lambda]^2$.
Assume that both populations are sufficiently large to admit a macroscopic description in terms of densities
$\rho_i:\mathbb{R}_{\ge 0}\times\Omega\to\mathbb{R}_{\ge 0}$ for $i\in\{L,F\}.$
The dynamics is subject to periodic boundary conditions, which ensure conservation of mass:
\begin{equation}
\label{eq:mass_conservation}
\int_\Omega \rho_i(t,\mathbf{x})\,\mathrm{d}\mathbf{x}
=
M_i,
\qquad
\forall t\ge0,\quad i\in\{L,F\},
\end{equation}
where \(M_i\) denotes the total mass of population \(i\). Without loss of generality, $M_i=N_i/(N_L+N_F)$, so that $M_L+M_F=1$, where $N_L$ and $N_F$ denote the numbers of leaders and followers.

The density dynamics follows the coupled PDEs given by 
\begin{subequations}\label{eq:LF_model}
\begin{align}
    &\partial_t \rho_L(t,\mathbf{x}) + \nabla \cdot [\rho_L(t,\mathbf{x})\,\mathbf{u}(t,\mathbf{x})] = 0,
    \label{eq: leader_evolution}\\
    &\partial_t \rho_F(t, \mathbf{x}) + \nabla \cdot [\rho_F(t, \mathbf{x})\mathbf{v}_{F}(t,\mathbf{x})] = D \; \Delta \rho_{F}(t, \mathbf{x}),
    \label{eq: follower_evolution}
\end{align}
\end{subequations}
where $\mathbf{u}:\mathbb{R}_{\geq 0} \times \Omega \to \mathbb{R}^2$ denotes the control velocity field, $D>0$ weights the diffusion strength, and $\mathbf{v}_F:\mathbb{R}_{\geq 0} \times \Omega \to \mathbb{R}^2$ models the leader-follower interactions. Specifically, we consider
\begin{equation}
    \mathbf v_F(t,\mathbf x)
    =
    \int_\Omega
    \mathbf g(\mathbf x-\mathbf y)\rho_L(t,\mathbf y)
    \,\mathrm d\mathbf y
    =
    (\mathbf g*\rho_L)(t,\mathbf x),
    \label{eq: interaction_field}
\end{equation}
where $\mathbf g$ is periodic with period $2\Lambda$ in each coordinate.
Although the formulation allows for different kinds of leader-follower interactions, here we consider a repulsive interaction obtained by periodizing the regularized kernel
\begin{equation}
    \frac{\mathbf x-\mathbf y}
    {\sqrt{\|\mathbf x-\mathbf y\|^2+\eta^2}}
    \exp\!\left(
        -\frac{\sqrt{\|\mathbf x-\mathbf y\|^2+\eta^2}}{L_r}
    \right).
    \label{eq:periodic_repulsive_kernel}
\end{equation}
Periodization is performed by summing the contributions of all copies of $\mathbf y$ translated by integer multiples of $2\Lambda$ in each coordinate. Here, $L_r>0$ denotes the interaction length scale and $\eta>0$ is a regularization parameter. The resulting periodic kernel $\mathbf g$ is bounded, odd, and smooth.

The model~\eqref{eq:LF_model} is initialized as $\rho_i(0,\cdot)=\rho_i^0$, $i\in\{L,F\}$.

The control input $\mathbf{u}$ in \eqref{eq: leader_evolution} is designed, following \citep{maffettone2025leader-follower}, to indirectly steer the follower density toward a desired profile $\bar{\rho}_F$. 
First, a target leader density $\bar\rho_L$ compatible with $\bar\rho_F$ is computed through a mass-preserving deconvolution procedure.
Then, the velocity field $\mathbf{u}$ yields exponential convergence of $\rho_L$ toward $\bar{\rho}_L$. Through the leader-follower interaction, this convergence induces the follower density $\rho_F$ to asymptotically reach the desired profile $\bar{\rho}_F$.

\subsection{Problem statement} 
\label{subsec:problem_statement}
Assume that the domain contains dangerous static regions represented, at the continuum level, by a spatial density 
$\rho_O:\Omega\to\mathbb{R}_{\ge 0}$. 
Let $R_k(\rho_i,\rho_O)$ denote the kernel-weighted overlap between population $i\in\{L,F\}$ and the dangerous regions, defined as
\begin{equation}
R_k(\rho_i,\rho_O)
=
\frac{1}{M_i}
\iint_{\Omega\times\Omega}
k(\mathbf x,\mathbf y)\rho_i(\mathbf x)\rho_O(\mathbf y)
\,\mathrm d\mathbf x\,\mathrm d\mathbf y.
\label{eq:Rk}
\end{equation}
where the kernel $k$ is obtained by periodizing the Gaussian kernel
\[
    \exp\!\left(
        -\frac{\|\mathbf x-\mathbf y\|^2}{2\sigma^2}
    \right),
\]
with $\sigma>0$, by summing its contributions over all periodic copies of $\mathbf y$. As proved in~\citep{tomaselli2026mean}, larger values of $R_k$ correspond to a higher spatial overlap between the two densities.

We assume that the initial density distributions of both populations satisfy
\begin{equation}
R_k(\rho_i^0,\rho_O)\leq \varepsilon_i,
\qquad i\in\{L,F\},
\label{eq:safe_initial_target}
\end{equation}
where $\varepsilon_i>0$ denotes the maximum admissible overlap between population $i$ and the dangerous regions.

The control objective is to steer the follower density toward a desired profile while keeping, throughout the entire evolution, the overlap of both populations with the dangerous regions under a given threshold.

To this end, we introduce the functionals
\begin{equation}
\mathcal{H}_i(\rho_i)
=
\varepsilon_i - R_k(\rho_i,\rho_O),
\qquad i\in\{L,F\},
\label{eq:barrier_functional}
\end{equation}
and define the associated safe sets $\mathcal{C}_i=\{\rho_i : \mathcal{H}_i(\rho_i) \geq 0\}$.
In the spirit of control barrier functions, our goal is to design the leader control input so as to render the sets $\mathcal{C}_i$, $i\in\{L,F\}$, forward invariant.

\begin{rem}[Nature of the safety requirement]
\label{rem:aggregate_safety}
The functionals \eqref{eq:barrier_functional} constrain the aggregate kernel-weighted overlap with the dangerous regions, which is the natural notion of safety at the mean-field level; they do not, by themselves, exclude that individual agents transiently enter those regions. The requirement becomes more stringent as $\varepsilon_i$ and $\sigma$ are reduced.
\end{rem}

\section{Higher-order Mean Field Control Barrier Functions}
\label{sec:ho-mf-cbf}
Consider a population described by a density $\rho(t,\mathbf{x})$ evolving on a domain $\Omega$ according to
\begin{equation}
\partial_t \rho = \mathcal{F}(\rho,\boldsymbol{\vartheta}),
\label{eq: general}
\end{equation}
where $\mathcal{F}$ denotes the density dynamics and $\boldsymbol{\vartheta}$ collects the variables through which the evolution of $\rho$ is influenced, including either the control input itself or other controlled state variables.
Assume that the dynamics preserve mass, so that
$\int_\Omega\mathcal F(\rho,\boldsymbol{\vartheta})\,\mathrm d\mathbf x=0$.

Let $\mathcal{H}(\rho)$ be a functional that defines the safe set
\begin{equation}
    \mathcal{C} = \{ \rho : \mathcal{H}(\rho) \geq 0 \}, \quad \partial\mathcal{C} = \{ \rho : \mathcal{H}(\rho) = 0 \}.
    \label{eq: safe_set_general}
\end{equation}
The total time derivative of $\mathcal{H}$ is given by 
\begin{equation}
\dot{\mathcal{H}}(\rho)
=
\int_\Omega \partial^f_\rho \mathcal{H}(\rho) \, \mathcal{F}(\rho, \boldsymbol{\vartheta})\, \mathrm{d}\mathbf{x}.
\end{equation}
In standard MF-CBF formulations \citep{fung2025mean}, the control input $\mathbf u$ appears explicitly in $\dot{\mathcal H}$, allowing one to enforce safety through a first-order barrier condition. However, in many systems, the control affects the state only indirectly and therefore does not appear explicitly in $\dot{\mathcal H}$. To formalize this scenario, we introduce the notion of relative degree.

\begin{defn}[Relative degree]
Consider the density dynamics \eqref{eq: general} and let $\mathcal{H}$ be a continuously differentiable functional. We say that $\mathcal{H}$ has relative degree $r$ with respect to the control input $\mathbf u$ if $r$ is the smallest integer such that the $r$-th total time derivative of $\mathcal{H}$, denoted by $\mathcal H^{(r)}$, depends explicitly on $\mathbf u$; that is, the coefficient multiplying $\mathbf u$ in $\mathcal H^{(k)}$ vanishes identically for all admissible states and all $k<r$, whereas for $k=r$ it does not vanish identically.
\end{defn}

The relative degree is thus a structural property of the pair $(\mathcal H,\mathcal F)$, in the same sense as in the finite-dimensional case and in~\citep{nurbekyan2026higher}. It identifies the derivative in which the control can enter, but it does not guarantee that the control coefficient is nonzero at every state reached along a given evolution. The latter condition, which we assume along the considered trajectory in Theorem~\ref{thm:graph_relative_degree} and Proposition~\ref{prop:follower_safety}, is what is needed for
the safety filter to be able to act.

Building on the concept of higher-order Control Barrier Functions \citep{xiao2021high}, we introduce the following notion of Higher-Order Mean-Field Control Barrier Function (HO-MF-CBF).

\begin{defn}[Higher-Order Mean-Field CBF]
Let $\mathcal H$ be a safety functional of relative degree $r$. Starting from $\psi_0(\rho) = \mathcal{H}(\rho)$, define the sequence of auxiliary functionals
\begin{equation}
\psi_k(\rho) = \dot{\psi}_{k-1}(\rho) + \alpha_{k-1}(\psi_{k-1}(\rho)), \quad k = 1, \dots, r-1
\label{eq: psik}
\end{equation}
where $\alpha_k:\mathbb R\to\mathbb R$, for $k=0,\ldots,r-1$, are extended class-$\mathcal K$ functions.
The functional $\mathcal H$ is said to be a Higher-Order Mean-Field Control Barrier Function if there exists a control input $\mathbf u$ such that
\begin{equation}
\dot{\psi}_{r-1}(\rho) + \alpha_{r-1}(\psi_{r-1}(\rho)) \geq 0.
\label{eq: condition_cbf}
\end{equation}
\end{defn}
The next result shows that condition \eqref{eq: condition_cbf} is sufficient to guarantee forward invariance of the safe set $\mathcal{C}$, provided that the auxiliary functionals are initially nonnegative.

\begin{thm}[Forward invariance of HO-MF-CBFs]
\label{thm:ho_mf_cbf_invariance}
Let $\mathcal{H}$ be a Higher-Order Mean-Field Control Barrier Function of relative degree $r$. Assume that, along the considered solution, $\psi_k(\rho(t,\cdot))$ is absolutely continuous in time for all $k=0,\ldots,r-1$.
Let $\rho^0(\mathbf{x}) = \rho(0,\mathbf{x})$ denote the density distribution at initial time $t=0$ and assume that
\[
\psi_k(\rho^0)\geq 0,
\qquad
k=0,\ldots,r-1.
\]
If the control input is chosen such that \eqref{eq: condition_cbf} holds for all $t\geq 0$, then the safe set $\mathcal{C}$ is forward invariant.
\end{thm}

\begin{pf}
Since $\psi_{r-1}(\rho^0)\ge0$ and the condition~\eqref{eq: condition_cbf} holds along the trajectory, the comparison lemma~\citep{khalil2002nonlinear} applied to the scalar function $\psi_{r-1}(\rho(t, \cdot))$, yields that $\psi_{r-1}(\rho(t, \cdot))\ge0$ for all $t\ge0$. By the recursive definition \eqref{eq: psik}, this implies
\[
\dot{\psi}_{r-2}+\alpha_{r-2}(\psi_{r-2})\geq 0,
\]
and hence $\psi_{r-2}(\rho(t, \cdot))\geq 0$ $\forall t\geq 0$. Repeating the same argument recursively yields $\psi_k(\rho)\geq 0$ for $k=0,\ldots,r-1$. In particular, since $\psi_0(\rho)=\mathcal{H}(\rho)$, we obtain $\mathcal{H}(\rho(t, \cdot))\geq 0$
$\forall t\geq 0$. Therefore, $\mathcal{C}$ is forward invariant. \qed
\end{pf}

\begin{rem}
For $r=1$, the proposed framework reduces to the standard MF-CBF formulations introduced in \citep{fung2025mean,tomaselli2026mean}.
\end{rem}

\begin{rem}
The recursive construction in Theorem~\ref{thm:ho_mf_cbf_invariance} mirrors finite-dimensional HO-CBFs~\citep{xiao2021high}, and the same recursion was derived independently in~\citep{nurbekyan2026higher}. There, the auxiliary functionals $\psi_k$ are expressed through the derivatives of $\mathcal H$ along the uncontrolled drift for any relative degree $r$, for a single population with control-affine agent dynamics. We do not claim the general recursion as a contribution. Our focus is on how relative degree greater than one arises from indirect, cross-population actuation, and on the resulting structure of the safety condition, established next.
\end{rem}

\subsection{Graph-induced relative degree in interconnected mean-field systems}
In several mean-field control problems, the control input does not act directly on the density appearing in the safety constraint. Instead, it transports another density, which in turn affects the evolution of the density of the population whose safety must be guaranteed. This situation naturally arises in the leader-follower density control problem described in Sec.~\ref{subsec:problem_statement}. More generally, indirect actuation can propagate through multiple interacting populations. We therefore characterize how the relative degree of a safety functional is determined by the directed interaction graph.

Consider $P$ interacting populations with densities $\rho_i:\mathbb{R}_{\geq 0}\times\Omega\rightarrow\mathbb{R}_{\geq 0}$, $i\in\mathcal{V}:=\{1,\ldots,P\}$, evolving on the periodic domain $\Omega$. Let $a\in\mathcal{V}$ denote the directly actuated population. The coupled density dynamics is assumed to have the form
\begin{equation}
    \partial_t \rho_i = \mathcal{F}_i(\boldsymbol{\rho}),
    \qquad i\neq a,
    \label{eq:interconnected_unactuated}
\end{equation}
and
\begin{equation}
    \partial_t \rho_a
    = \mathcal{F}_a(\boldsymbol{\rho})
    -\nabla\cdot(\rho_a \mathbf u),
    \label{eq:interconnected_actuated}
\end{equation}
where $\boldsymbol{\rho}:=(\rho_1,\ldots,\rho_P)$, and $\mathbf u:\mathbb{R}_{\geq0}\times\Omega\rightarrow\mathbb{R}^2$
is the control velocity field.

The interdependence among the population dynamics is represented by a directed graph $\mathcal{G}=(\mathcal{V},\mathcal{E})$ where an edge $(j,i)\in\mathcal E$ is present whenever $\mathcal F_i$ depends on $\rho_j$.

Let $\mathcal H(\rho_s)$ be a continuously Fr\'echet differentiable safety functional associated with population $s\in\mathcal{V}$. Denote by $d_\mathcal{G}(j,s)$ the length of the shortest directed path from node $j$ to node $s$, with $d_\mathcal{G}(s,s)=0$, and define the set
\begin{equation}
    S_k(s)
    :=
    \left\{
        j\in\mathcal{V}:
        d_\mathcal{G}(j,s)\leq k
    \right\}.
    \label{eq:Sk_definition}
\end{equation}
The following result characterizes how the dependence of the safety functional propagates through the interaction graph.

\begin{lem}[Propagation of functional dependence]
\label{lem:graph_dependency}
Let $\mathcal H(\rho_s)$ be a safety functional associated with population $s\in\mathcal V$, and assume that the density trajectories and the functionals involved are sufficiently regular so that all time derivatives up to order $d_\mathcal{G}(a,s)+1$, the corresponding first variations, and the spatial derivatives required by the integrations by parts below exist.
For every $k\leq d_\mathcal{G}(a,s)$, the $k$-th time derivative of $\mathcal H$ can depend only on populations whose directed distance to $s$ is at most $k$, namely
\begin{equation}
    \mathcal H^{(k)}
    =
    \mathcal H^{(k)}
    \left(
        \{\rho_j\}_{j\in S_k(s)}
    \right).
    \label{eq:Hk_dependency}
\end{equation}
\end{lem}
\begin{pf}
The result follows by induction. For $k=0$, $\mathcal H=\mathcal H(\rho_s)$, and $S_0(s)=\{s\}$. Assume that \eqref{eq:Hk_dependency} holds for some $k<d_\mathcal{G}(a,s)$.
By the chain rule, differentiating along the density dynamics yields
\begin{equation}
    \mathcal H^{(k+1)}
    =
    \sum_{j\in S_k(s)}
    \int_\Omega
    \partial_{\rho_j}^{f} \mathcal  H^{(k)}
    \,
    \partial_t\rho_j\, \mathrm d \mathbf x.
    \label{eq:Hkplus1_dependency}
\end{equation}
Since $k<d_\mathcal{G}(a,s)$, the actuated population $a$ does not belong to $S_k(s)$. Hence, substituting the population dynamics gives
\begin{equation}
    \mathcal H^{(k+1)}
    =
    \sum_{j\in S_k(s)}
    \int_\Omega
    \partial_{\rho_j}^{f} \mathcal H^{(k)}
    \,
    \mathcal F_j(\boldsymbol{\rho})\,\mathrm d \mathbf x.
    \label{eq:Hkplus1_uncontrolled}
\end{equation}
By definition of the interaction graph, each $\mathcal{F}_j$ depends only on population $j$ and on populations directly influencing $j$. Consequently, the right-hand side can depend only on populations whose directed distance to $s$ is at most $k+1$, i.e., $\mathcal H^{(k+1)}=\mathcal H^{(k+1)}(\{\rho_j\}_{j\in S_{k+1}(s)})$. \qed
\end{pf}
The graph structure can now be directly related to the relative degree of the safety functional.

\begin{thm}[Graph-induced relative degree]
\label{thm:graph_relative_degree}
Consider the interconnected mean-field system \eqref{eq:interconnected_unactuated}-\eqref{eq:interconnected_actuated}.
Let $\mathcal H(\rho_s)$ be a safety functional associated with population $s\in\mathcal V$, satisfying the regularity assumptions of Lemma~\ref{lem:graph_dependency}. Assume that there exists a directed path from $a$ to $s$, and let $d:=d_\mathcal{G}(a,s)$.
Then the control input $\mathbf u$ does not appear explicitly in $\mathcal H^{(k)}$, $k=0,\ldots,d$.
Moreover,
\begin{equation}
    \mathcal H^{(d+1)}
    =
    I_d(\boldsymbol{\rho})
    +
    \int_\Omega
    \boldsymbol\Psi_d(\boldsymbol{\rho})
    \cdot \mathbf u(t, \mathbf{x})\,\mathrm d \mathbf x,
    \label{eq:Hdplus1_affine}
\end{equation}
where
\begin{equation}
    I_d(\boldsymbol{\rho})
    :=
    \sum_{j\in S_d(s)}
    \int_\Omega
    \partial_{\rho_j}^{f}\mathcal H^{(d)}
    \,
    \mathcal F_j(\boldsymbol{\rho})\,\mathrm d \mathbf x,
    \label{eq:Id_graph}
\end{equation}
and
\begin{equation}
    \boldsymbol\Psi_d(\boldsymbol{\rho})
    :=
    \rho_a
    \nabla
    \partial_{\rho_a}^{f}\mathcal H^{(d)}.
    \label{eq:Psi_graph}
\end{equation}
If $\boldsymbol\Psi_d(\boldsymbol{\rho})\not\equiv0$ along the considered trajectory, then $\mathcal H$ has relative degree
$r_{\mathcal H}=d+1$
with respect to $\mathbf u$.
\end{thm}

\begin{pf}
Since $a\notin S_k(s)$ for $k<d$, Lemma~\ref{lem:graph_dependency} implies that $\mathcal H^{(k)}$ does not depend on $\rho_a$ for $k<d$. As $\mathbf u$ enters the system only through the dynamics of $\rho_a$, it does not appear explicitly in $\mathcal H^{(k)}$ for $k=0,\ldots,d$.
At order $d$, $a\in S_d(s)$, so $\mathcal H^{(d)}$ may depend on $\rho_a$. Differentiating once more as in \eqref{eq:Hkplus1_dependency}, substituting \eqref{eq:interconnected_unactuated}-\eqref{eq:interconnected_actuated}, and integrating by parts the transport term of population $a$ on the periodic domain,
\[
    -\int_\Omega
    \partial_{\rho_a}^{f}\mathcal H^{(d)}
    \,
    \nabla\cdot(\rho_a \mathbf u)\,\mathrm d \mathbf x
    =
    \int_\Omega
    \rho_a
    \nabla
    \partial_{\rho_a}^{f}\mathcal H^{(d)}
    \cdot \mathbf  u\,\mathrm d \mathbf x,
\]
yields \eqref{eq:Hdplus1_affine}. Hence, if $\boldsymbol\Psi_d\not\equiv0$, the control appears explicitly for the first time in $\mathcal H^{(d+1)}$, and $r_{\mathcal H}=d+1$. \qed
\end{pf}

\begin{cor}[Affine HO-MF-CBF constraint]
\label{cor:graph_affine_constraint}
Under the assumptions of Theorem~\ref{thm:graph_relative_degree}, let $d=d_\mathcal{G}(a,s)$, let $\alpha_k:\mathbb R\to\mathbb R$, $k=0,\ldots,d$, be extended class-$\mathcal K$ functions with $\alpha_k$ of class $C^{d-k}$, and define $\psi_0=\mathcal H$ and
\[
    \psi_k=\dot\psi_{k-1}
    +\alpha_{k-1}(\psi_{k-1}),
    \qquad k=1,\ldots,d.
\]
Then there exists a function $\Gamma_d$ of $(\mathcal H,\dot{\mathcal H},\ldots,\mathcal H^{(d)})$ only, and hence independent of $\mathbf u$, such that the HO-MF-CBF condition $\dot\psi_d+\alpha_d(\psi_d)\geq 0$ is equivalent to the affine inequality
\begin{equation}
\int_\Omega \boldsymbol\Psi_d(\boldsymbol\rho)\cdot \mathbf u(t,\mathbf x)\,\mathrm d\mathbf x
    \geq
    -I_d(\boldsymbol\rho)
    -\Gamma_d
    -\alpha_d(\psi_d).
\label{eq:graph_affine_constraint}
\end{equation}
Consequently, if $\psi_k(0)\geq0$ for $k=0,\ldots,d$, and $\mathbf u$ satisfies \eqref{eq:graph_affine_constraint} for all $t\geq0$, then $\mathcal C=\{\rho_s:\mathcal H(\rho_s)\geq0\}$ is forward invariant.
\end{cor}
\begin{pf}
By induction on $k$, one has
\[
    \psi_k=\mathcal H^{(k)}+\phi_k,
\]
where $\phi_k$ depends only on $\mathcal H,\ldots,\mathcal H^{(k-1)}$.
Moreover, the regularity of $\alpha_0,\ldots,\alpha_{k-1}$ ensures that $\phi_k$ is of class $C^{d-k+1}$.

Differentiating $\psi_d$ along the trajectories then gives
\[
    \dot\psi_d=\mathcal H^{(d+1)}+\Gamma_d,
    \qquad
    \Gamma_d=
    \sum_{i=0}^{d-1}
    \partial_i\phi_d\,\mathcal H^{(i+1)}.
\]
The term $\Gamma_d$ depends only on $\mathcal H,\ldots,\mathcal H^{(d)}$ and is therefore independent of $\mathbf u$ by
Theorem~\ref{thm:graph_relative_degree}. Substituting \eqref{eq:Hdplus1_affine} into the expression for $\dot\psi_d$ yields \eqref{eq:graph_affine_constraint}. Finally, forward invariance follows from Theorem~\ref{thm:ho_mf_cbf_invariance} with $r=d+1$.
\qed
\end{pf}
For $d=1$, $\Gamma_1=\alpha_0'(\mathcal H)\,\dot{\mathcal H}$, and \eqref{eq:graph_affine_constraint} reads
\[
\int_\Omega \boldsymbol\Psi_1\cdot\mathbf u\,\mathrm d\mathbf x
\ge
-I_1-\alpha_0'(\mathcal H)\,\dot{\mathcal H}-\alpha_1\big(\dot{\mathcal H}+\alpha_0(\mathcal H)\big),
\]
which is the form used for the follower population in Sec.~\ref{sec:lf_safety_constraints}.

\begin{rem}[Single- vs. cross-population actuation]
In \citep{nurbekyan2026higher}, relative degree two arises for a single population whose agents follow $\dot{\mathbf x}=\mathbf f+g\mathbf u$ on a phase space, the safety functional depends only on the position marginal, and the control enters through the velocity components. In Theorem~\ref{thm:graph_relative_degree}, by contrast, the constrained population and the actuated population may be distinct and may obey different dynamics. In particular, the control acts on the constrained population only through the inter-population dependencies encoded by the interaction graph, and the control coefficient $\boldsymbol\Psi_d(\boldsymbol\rho)=\rho_a\nabla\partial_{\rho_a}^f\mathcal H^{(d)}$ is supported on the actuated population.
\end{rem}

\section{Safety Constraints for the Leader-Follower Mean-Field Model}
\label{sec:lf_safety_constraints}

We now specialize the HO-MF-CBF framework to the leader-follower model introduced in Section~\ref{subsec:problem_statement}, with the safety functionals associated with the dangerous regions defined in \eqref{eq:barrier_functional}.

\subsection{Follower safety}
Since the control affects $\rho_F$ only through $\rho_L$, via the interaction field~\eqref{eq: interaction_field}, the follower safety functional has relative degree higher than one, as formalized next.

\begin{prop}[Follower safety constraint]
\label{prop:follower_safety}
Consider the leader-follower model~\eqref{eq:LF_model}
on the periodic domain $\Omega$, with sufficiently smooth periodic solutions and a continuously differentiable interaction kernel
$\mathbf g$ with bounded Jacobian. Let
\[
\mathcal H_F(\rho_F)=\varepsilon_F-R_k(\rho_F,\rho_O),
\qquad
\mathcal C_F=\{\rho_F:\mathcal H_F(\rho_F)\ge0\}.
\]
Assume that $\mathcal H_F$ is continuously Fr\'echet differentiable, $\partial^f_{\rho_F}\mathcal H_F$ is spatially differentiable, and that the coefficient
\begin{equation}
\boldsymbol\Psi_F
=\rho_L(t,\mathbf y)\int_\Omega \rho_F(t,\mathbf x) J_{\mathbf y}\mathbf g(\mathbf x-\mathbf y)^\top \nabla\partial^f_{\rho_F}\mathcal H_F(\rho_F) \,\mathrm d\mathbf x
\label{eq:psi_F}
\end{equation}
is not identically zero along the considered trajectory. Then $\mathcal H_F$ has relative degree two with respect to the leader control field $\mathbf u$, and
\begin{equation}
\ddot{\mathcal H}_F = I_F(\rho_F,\rho_L) + \int_\Omega \boldsymbol\Psi_F\cdot\mathbf u(t,\mathbf y)\,\mathrm d\mathbf y,
\label{eq:follower_second_derivative_decomposition}
\end{equation}
where $I_F$ collects all control-independent terms.

In addition, let $\alpha_{F,0},\alpha_{F,1}:\mathbb R\to\mathbb R$ be continuously differentiable extended class-$\mathcal K$ functions and define the set 
\begin{equation}
\begin{aligned}
\mathcal K_F := \Big\{ \mathbf{u}:{} & \int_\Omega
\boldsymbol\Psi_F(\mathbf y)\cdot\mathbf u(\mathbf y)\,\mathrm d\mathbf y
\ge
-I_F(\rho_F,\rho_L) \\
& -\dot \alpha_{F,0}(\mathcal H_F)
-\alpha_{F,1}
\big(\dot{\mathcal H}_F+\alpha_{F,0}(\mathcal H_F) \big) \Big\}.
\end{aligned}
\label{eq:follower_safe_set}
\end{equation}
where $\dot\alpha_{F,0}(\mathcal H_F)$ denotes the total time derivative of $\alpha_{F,0}(\mathcal H_F)$, i.e., $\alpha_{F,0}'(\mathcal H_F)\,\dot{\mathcal H}_F$ for time-invariant $\alpha_{F,0}$; for the time-varying gains of Sec.~\ref{subsec:feasibility}, it also includes $\partial_t\alpha_{F,0}$.
Assume that
\begin{equation}
\mathcal H_F(\rho_F^0)\ge 0,
\qquad
\dot{\mathcal H}_F(\rho_F^0,\rho_L^0) + \alpha_{F,0}\big(\mathcal H_F(\rho_F^0)\big) \ge 0.
\label{eq:initial_conditions}
\end{equation}
If $\mathbf{u} \in \mathcal K_F$ for all \(t\ge0\), then \(\mathcal C_F\) is forward invariant.
\end{prop}
\begin{pf}
The follower dynamics is a special case of the interconnected mean-field system of Theorem~\ref{thm:graph_relative_degree}, with the follower population as the constrained population and the leader population as the actuated one.
The leader density satisfies
\[
\partial_t\rho_L=-\nabla\cdot(\rho_L\mathbf u).
\]
By integration by parts, given the periodic boundary conditions of $\Omega$, the first
time derivative of \(\mathcal H_F\) is
\begin{equation}
\begin{aligned}
\dot{\mathcal H}_F
&=
\int_\Omega \nabla\partial^f_{\rho_F}\mathcal H_F(\rho_F) \cdot \rho_F(t,\mathbf x)\mathbf v_F(t,\mathbf x) \,\mathrm d\mathbf x \\
&\quad+ D\int_\Omega \Delta\partial^f_{\rho_F}\mathcal H_F(\rho_F) \rho_F(t,\mathbf x) \,\mathrm d\mathbf x ,
\end{aligned}
\label{eq:HF_dot}
\end{equation}
where
\[
\mathbf v_F(t,\mathbf x) = \int_\Omega \mathbf g(\mathbf x-\mathbf y) \rho_L(t,\mathbf y)\,\mathrm d\mathbf y .
\]
Thus \(\dot{\mathcal H}_F\) depends on \(\rho_L\), but not explicitly
on \(\mathbf u\).

The diffusion term in \eqref{eq:HF_dot} does not depend on $\rho_L$. Hence
\[
\partial^f_{\rho_L}\dot{\mathcal H}_F
= \int_\Omega \rho_F(t,\mathbf x) \nabla\partial^f_{\rho_F}\mathcal H_F(\rho_F) \cdot
\mathbf g(\mathbf x-\mathbf y) \,\mathrm d\mathbf x .
\]
Taking the spatial gradient with respect to $\mathbf y$ gives
\[
\nabla_{\mathbf y}\partial^f_{\rho_L}\dot{\mathcal H}_F = \int_\Omega \rho_F(t,\mathbf x)
J_{\mathbf y}\mathbf g(\mathbf x-\mathbf y)^\top \nabla\partial^f_{\rho_F}\mathcal H_F(\rho_F) \,\mathrm d\mathbf x .
\]
By Theorem~\ref{thm:graph_relative_degree}, since the leader population is at directed distance one from the follower population, the follower safety functional has relative degree two with respect to $\mathbf u$, and
\[
\ddot{\mathcal H}_F = I_F(\rho_F,\rho_L) + \int_\Omega \boldsymbol\Psi_F(\mathbf y)\cdot\mathbf u(t,\mathbf y)\,\mathrm d\mathbf y,
\]
where $I_F$ collects all the control-independent terms and 
\[
\boldsymbol\Psi_F(\mathbf{y}) = \rho_L(t,\mathbf y) \nabla_{\mathbf y} \partial^f_{\rho_L}\dot{\mathcal H}_F(\rho_F, \rho_L),
\]
which gives \eqref{eq:psi_F}. Since, by assumption, $\boldsymbol\Psi_F$ is not identically zero, the relative-degree-two condition is nondegenerate along the considered trajectory.

Hence, the constraint $\mathbf u\in \mathcal K_F$ is precisely the affine HO-MF-CBF condition of Corollary~\ref{cor:graph_affine_constraint} with $d=1$.
Therefore, under the initial conditions~\eqref{eq:initial_conditions}, Theorem~\ref{thm:ho_mf_cbf_invariance} guarantees that $\mathcal C_F$ is forward invariant.
\qed
\end{pf}
For completeness, since $\mathcal H_F$ is affine in $\rho_F$ and the dangerous regions are static, the control-independent term in~\eqref{eq:follower_second_derivative_decomposition} is
\begin{equation}
I_F=\int_\Omega\big(\nabla\phi_F\cdot\mathbf v_F+D\Delta\phi_F\big)\big(D\Delta\rho_F-\nabla\cdot(\rho_F\mathbf v_F)\big)\,\mathrm d\mathbf x,
\label{eq:Ir_expression}
\end{equation}
with $\phi_F:=\partial^f_{\rho_F}\mathcal H_F$; after integration by parts, it can be evaluated directly on empirical densities.

\subsection{Leader safety}
Unlike the follower population, the leader population is directly actuated by the control field $\mathbf u$. As a result, the leader safety functional has relative degree one, as formalized in the following proposition.

\begin{prop}[Leader safety constraint]
\label{prop:leader_safety}
Let
\[
\mathcal H_L(\rho_L) = \varepsilon_L-R_k(\rho_L,\rho_O).
\]
Assume that $\rho_L$, $\mathbf u$, and $\mathcal H_L$ are sufficiently regular and periodic so that the differentiation of $\mathcal H_L$ along the leader dynamics and the integration by parts below are well defined.

Then $\mathcal H_L$ has relative degree one with respect to the control field $\mathbf u$, and
\begin{equation}
\dot{\mathcal H}_L = \int_\Omega \rho_L(t,\mathbf x) \nabla\partial^f_{\rho_L}\mathcal H_L(\rho_L) \cdot \mathbf u(t,\mathbf x) \,\mathrm d\mathbf x .
\label{eq:H_L_dot}
\end{equation}
In addition, let $\alpha_L:\mathbb R\to\mathbb R$ be an extended class-$\mathcal K$ function and define
\begin{equation}
\mathcal K_L = \left\{
\mathbf u: \int_\Omega \rho_L \nabla\partial^f_{\rho_L}\mathcal H_L(\rho_L) \cdot \mathbf u \,\mathrm d\mathbf x
\ge
-\alpha_L(\mathcal H_L) \right\}.
\label{eq:leader_safe_set}
\end{equation}
Assuming that 
\begin{equation}
    \mathcal H_L(\rho_L^0)\ge0
\label{eq:initial_condition_leader}
\end{equation}
if $\mathbf u(t,\cdot)\in\mathcal K_L$  for all $t\ge0$, then $\mathcal C_L$ is forward invariant.
\end{prop}
\begin{pf}
Equation~\eqref{eq:H_L_dot} follows from $\dot{\mathcal H}_L=-\int_\Omega \partial^f_{\rho_L}\mathcal H_L\,\nabla\cdot(\rho_L\mathbf u)\,\mathrm d\mathbf x$ by integration by parts on the periodic domain, and is the case $d=0$ of Theorem~\ref{thm:graph_relative_degree}. Thus, the control appears explicitly in the first derivative of $\mathcal H_L$, the relative degree is one, and $\mathbf u\in\mathcal K_L$ is equivalent to $\dot{\mathcal H}_L+\alpha_L(\mathcal H_L)\ge0$.
The set invariance of $\mathcal C_L$ follows directly from Theorem~\ref{thm:ho_mf_cbf_invariance} with relative degree $r=1$, since $\mathcal H_L(\rho_L^0)\ge0$.
\qed
\end{pf}

Given a nominal leader control input $\mathbf u$, the safety-adjusted control is computed by solving
\begin{equation}
\begin{aligned}
\mathbf u^\star 
= \arg\min_{\mathbf u' \in \mathcal U} 
\quad& 
\|\mathbf u'-\mathbf u\|_2^2
\\
\text{s.t.}\quad&
\mathbf u'\in \mathcal K_F \cap \mathcal K_L .
\end{aligned}
\label{eq:qp_leader_follower}
\end{equation}
where $\mathcal U$ denotes the set of periodic velocity fields, and $\mathcal{K}_L$ and $\mathcal{K}_F$ are defined in \eqref{eq:leader_safe_set} and \eqref{eq:follower_safe_set}, respectively. 
Since both $\mathcal K_F$ and $\mathcal K_L$ are defined by affine inequalities in $\mathbf u'$, problem~\eqref{eq:qp_leader_follower} is a convex quadratic program whenever the admissible set is nonempty.

As a consequence of Propositions~\ref{prop:follower_safety} and~\ref{prop:leader_safety}, if the quadratic program~\eqref{eq:qp_leader_follower} is feasible for all $t\ge0$ and the conditions \eqref{eq:initial_conditions} and \eqref{eq:initial_condition_leader} are initially satisfied, then the set $\mathcal C_L\cap\mathcal C_F$ is forward invariant.

\subsection{Feasibility of simultaneous leader-follower safety constraints}
\label{subsec:feasibility}
To make explicit the structure of the safety filter, rewrite the follower and leader constraints as
\begin{equation}
\langle \mathbf{a}_F,\mathbf u'\rangle
\ge b_F,
\qquad
\langle \mathbf{a}_L,\mathbf u'\rangle
\ge b_L,
\label{eq:affine_constraints_general}
\end{equation}
where
\[
\begin{aligned}
\mathbf a_F &= \boldsymbol{\Psi}_F,
\qquad
\mathbf a_L = \rho_L\nabla\partial^f_{\rho_L}\mathcal H_L,
\qquad
b_L = -\alpha_L(\mathcal H_L),\\
b_F &= -\left[
I_F+\dot{\alpha}_{F,0}(\mathcal H_F)
+\alpha_{F,1}\left(
\dot{\mathcal H}_F+\alpha_{F,0}(\mathcal H_F)
\right)\right].
\end{aligned}
\]
Whenever $\mathcal K_F\cap\mathcal K_L\neq\emptyset$, the problem \eqref{eq:qp_leader_follower} has the closed-form solution
$\mathbf u^\star=\mathbf u+\mu_F\mathbf a_F+\mu_L\mathbf a_L$, with $\mu_F,\mu_L\geq0$. Specifically, if only the constraint $i\in\{F,L\}$ is active, then $\mu_i=(b_i-\langle\mathbf a_i,\mathbf u\rangle)/\|\mathbf a_i\|_2^2$ and the other multiplier vanishes; if both are active and $\mathbf a_F,\mathbf a_L$ are linearly independent, $(\mu_F,\mu_L)$ solves the $2\times2$ Gram system $\langle\mathbf a_i,\mathbf a_F\rangle\mu_F+\langle\mathbf a_i,\mathbf a_L\rangle\mu_L=b_i-\langle\mathbf a_i,\mathbf u\rangle$, $i\in\{F,L\}$.

The admissible set $\mathcal K_F\cap\mathcal K_L$ may be empty when the two requirements conflict, for instance when keeping the followers away from dangerous regions requires moving the leaders toward regions that are unsafe for them. If $\mathbf a_F=-c\,\mathbf a_L$ with $c>0$, the constraints~\eqref{eq:affine_constraints_general} reduce to $b_L\le\langle\mathbf a_L,\mathbf u'\rangle\le-b_F/c$, which are compatible if and only if
\begin{equation}
c \,b_L\le -b_F.
\label{eq:worst_case_feasibility}
\end{equation}
The next result shows that, when $\mathcal U$ imposes no bounds on the control field, this is the only nondegenerate source of infeasibility.

\begin{prop}[Exact feasibility characterization]
\label{prop:exact_feasibility}
Consider the two constraints \eqref{eq:affine_constraints_general} and assume $\mathcal H_L\geq0$, so that $b_L\leq0$.
Then the two constraints are simultaneously feasible for every configuration of $\mathbf a_F$ and $\mathbf a_L$, except in either of the following cases: (i) $\mathbf a_F\equiv0$ and $b_F>0$; (ii) the two constraint directions are antiparallel, i.e.,
    \[
        \mathbf a_F=-c\,\mathbf a_L\neq0,
        \qquad c>0,
    \]
    and their bounds are incompatible, namely
    $
        c b_L>-b_F.
    $
Equivalently,
\[
    \mathcal K_F\cap\mathcal K_L=\emptyset
\]
if and only if either (i) or (ii) holds.
In particular, if $b_F\leq0$, then $\mathcal K_F\cap\mathcal K_L\neq\emptyset$, since $\mathbf u'=0$ satisfies both constraints.
\end{prop}
\begin{pf}
Suppose first that $\mathbf a_F$ and $\mathbf a_L$ are linearly independent. Their Gram matrix is then positive definite and hence invertible, so there exists $\mathbf u'\in\operatorname{span}\{\mathbf a_F,\mathbf a_L\}$ with $\langle\mathbf a_F,\mathbf u'\rangle=b_F$ and $\langle\mathbf a_L,\mathbf u'\rangle=b_L$. 

Assume now that $\mathbf a_F$ and $\mathbf a_L$ are linearly dependent. 

If $\mathbf a_L\equiv0$, the leader constraint is always satisfied since $b_L\leq0$. The follower constraint is feasible whenever $\mathbf a_F\not\equiv0$, whereas, if $\mathbf a_F\equiv0$, it reduces to $0\geq b_F$ and is infeasible if and only if $b_F>0$, yielding case~(i).

Now suppose that $\mathbf a_L\not\equiv0$. If $\mathbf a_F\equiv0$, the follower constraint is infeasible if and only if $b_F>0$, again yielding case~(i). Otherwise, since $\mathbf a_F$ and $\mathbf a_L$ are linearly dependent, there exists $c>0$ such that $\mathbf a_F=c\,\mathbf a_L$ or $\mathbf a_F=-c\,\mathbf a_L$.

If $\mathbf a_F=c\,\mathbf a_L$, both constraints impose lower bounds along the same direction and are therefore simultaneously feasible. If $\mathbf a_F=-c\,\mathbf a_L$, \eqref{eq:worst_case_feasibility} shows that the constraints are infeasible if and only if $cb_L>-b_F$, which yields case~(ii). \qed
\end{pf}

Proposition~\ref{prop:exact_feasibility} shows that, in the absence of bounds on the control field, apart from the degenerate case $\mathbf a_F\equiv0$, infeasibility emerges in the case of antiparallel alignment of the two constraint directions. Nevertheless, near-antiparallel configurations may require large corrections of the nominal control even when $\mathcal K_F\cap\mathcal K_L\neq\emptyset$.

Inspired by \citep{xiao2021adaptive,parwana2025rate}, we recover feasibility in the antiparallel case by adapting $\alpha_{F,0}$.
Let $\alpha_{F,0}(t,s)=\gamma_0(t)s$, $\alpha_{F,1}(s)=\gamma_1s$, and $\alpha_L(s)=\gamma_Ls$, where $\gamma_1,\gamma_L>0$ are constant and $\dot\gamma_0=\nu_0$. Since $\psi_{F,1}=\dot{\mathcal H}_F+\gamma_0\mathcal H_F$, its derivative contains the additional term $\nu_0\mathcal H_F$, yielding the feasibility condition
\begin{equation}
B+\nu_0\mathcal H_F\geq0,
\label{eq:feasibility_condition}
\end{equation}
where $B:=I_F+(\gamma_0+\gamma_1)\dot{\mathcal H}_F+\gamma_0\gamma_1\mathcal H_F+c\gamma_L\mathcal H_L$.

To preserve $\gamma_0(t)\geq\gamma_0(0)>0$, we impose $\nu_0\geq-\beta(\gamma_0-\gamma_0(0))$, where $\beta$ is a locally Lipschitz extended class-$\mathcal K$ function, and minimize $\nu_0$ pointwise subject to both constraints. This choice keeps the gain at its initial value whenever feasible and allows it to decrease toward that value after an increase, whenever the feasibility constraint permits.
\begin{prop}[Adaptive feasibility]
\label{prop:adaptive_feasibility}
Assume $\mathcal H_F>0$,
$\mathbf a_F=-c\mathbf a_L\neq\mathbf 0$ for some $c>0$,
and $\gamma_0(0)>0$.
The smallest admissible update is
\begin{equation}
\nu_0^\star=\max\left\{
-\frac{B}{\mathcal H_F},
-\beta(\gamma_0-\gamma_0(0))
\right\}.
\label{eq:opt_feasibility}
\end{equation}
This update renders the safety-filter QP~\eqref{eq:qp_leader_follower} feasible and preserves $\gamma_0(t)\geq\gamma_0(0)>0$ along the solution.
\end{prop}
\begin{pf}
Since $\mathcal H_F>0$, condition
\eqref{eq:feasibility_condition} and the barrier constraint give $\nu_0\geq\max\{-B/\mathcal H_F,-\beta(\gamma_0-\gamma_0(0))\}$.
This interval is nonempty, and its smallest element is \eqref{eq:opt_feasibility}. Feasibility of the safety-filter QP follows from Proposition~\ref{prop:exact_feasibility}. Finally, $\dot\gamma_0=\nu_0^\star \geq-\beta(\gamma_0-\gamma_0(0))$ and the comparison lemma guarantee $\gamma_0(t)\geq\gamma_0(0)>0$ along the solution.
\qed
\end{pf}

When $\mathcal H_F=0$, the adaptive term vanishes and feasibility requires $B\geq0$. The forward invariance argument of Theorem~\ref{thm:ho_mf_cbf_invariance} extends to $\alpha_{F,0}(t,s)=\gamma_0(t)s$ by the same comparison argument.
 
\begin{figure*}[htb]
\centering
\subfloat[]{\includegraphics[width=0.19\textwidth]{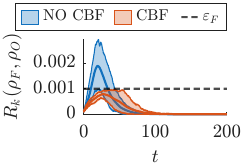}\label{fig:temporal_evolution}}\hfill
\subfloat[]{\includegraphics[width=0.19\textwidth]{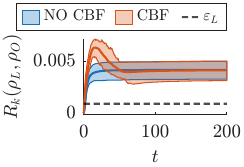}\label{fig:total}}\hfill
\subfloat[]{\includegraphics[width=0.19\textwidth]{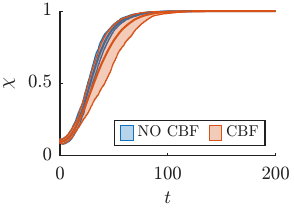}\label{fig: chi}}\hfill
\subfloat[]{\includegraphics[width=0.14\textwidth]{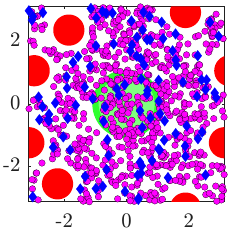}\label{fig:start}}\hfill
\subfloat[]{\includegraphics[width=0.14\textwidth]{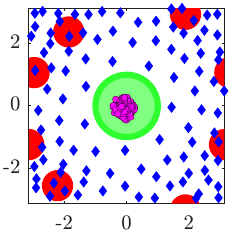}\label{fig:end}}
\caption{Follower safety only: constrained (orange) vs.\ unconstrained (blue) case over 50 simulations. Kernel-weighted overlap of (a) followers and (b) leaders with the dangerous regions; (c) fraction $\chi$ of followers within the goal region. Representative example with the second-order filter at (d) $t=0$ and (e) $t=T$: leaders (blue diamonds), followers (magenta dots), goal region (green disk).}
\end{figure*}
\begin{figure*}[htb]
\centering
\subfloat[]{\includegraphics[width=0.19\textwidth]{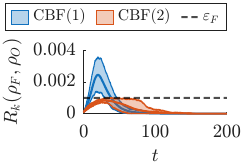}\label{fig:temporal_evolution2}}\hfill
\subfloat[]{\includegraphics[width=0.19\textwidth]{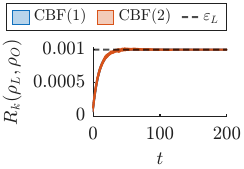}\label{fig:total2}}\hfill
\subfloat[]{\includegraphics[width=0.19\textwidth]{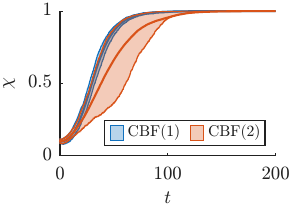}\label{fig:chi2}}\hfill
\subfloat[]{\includegraphics[width=0.15\textwidth]{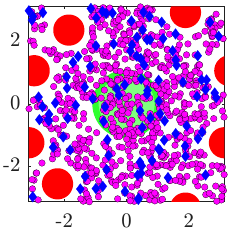}\label{fig:start2}}\hfill
\subfloat[]{\includegraphics[width=0.15\textwidth]{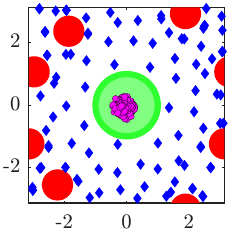}\label{fig:end2}}
\caption{Leader and follower safety: first-order leader-only filter CBF(1) (blue) vs.\ proposed second-order filter CBF(2) (orange) over 50 simulations. Kernel-weighted overlap of (a) followers and (b) leaders with the dangerous regions; (c) fraction $\chi$ of followers within the goal region. Representative example with CBF(2) at (d) $t=0$ and (e) $t=T$, with all the followers within the goal region.}
\end{figure*}

\subsection{Tracking error under safety-filtered leader dynamics}
In the absence of dangerous regions, the nominal control $\mathbf u$ steers $\rho_L$ toward the target $\bar\rho_L$, and the induced interaction field $\bar{\mathbf v}_F:=\mathbf g*\bar\rho_L$ steers $\rho_F$ toward $\bar\rho_F$. Under the filtered control~\eqref{eq:qp_leader_follower}, the leader density may deviate from $\bar\rho_L$, which in turn modifies the transport field acting on the followers, so that exact convergence of $\rho_F$ to $\bar\rho_F$ cannot be guaranteed. The next result quantifies this effect.

\begin{prop}[Tracking error under the safety filter]
\label{prop:tracking}
Consider the leader-follower model~\eqref{eq:LF_model}
with the nominal mass-preserving control $\mathbf u$
of~\citep{maffettone2025leader-follower}.
Let $\mathbf u^\star$ be generated
by~\eqref{eq:qp_leader_follower},
$\delta\mathbf u:=\mathbf u^\star-\mathbf u$,
and let $\rho_L^\star$, $\rho_F$ denote the resulting densities.
Define $e_L:=\bar\rho_L-\rho_L^\star$,
$e_F:=\bar\rho_F-\rho_F$, and let
\[
    V_L:=\tfrac12\|e_L\|_2^2,
    \qquad
    V_F:=\tfrac12\int_\Omega
    \frac{e_F^2}{\bar\rho_F}\,\mathrm d\mathbf x.
\]
Assume that:
(i) $\bar\rho_L$ is an equilibrium of the nominal
closed-loop leader dynamics and there exists $\lambda_L>0$
such that
\[\langle\bar\rho_L-\rho_L,
\nabla\cdot(\rho_L\mathbf u)\rangle
\leq-\frac{\lambda_L}{2}\|\bar\rho_L-\rho_L\|_2^2\]
for every sufficiently regular admissible density $\rho_L$,
with $\mathbf u$ evaluated at $\rho_L$;
(ii) the target follower density is smooth, periodic,
satisfies $0<\underline{q}_F\leq\bar\rho_F\leq \overline q_F$, and obeys
\[-\nabla\cdot(\bar\rho_F\bar{\mathbf v}_F)
+D\Delta\bar\rho_F=0,\] where
$\bar{\mathbf v}_F=\mathbf g*\bar\rho_L$.
Let $\lambda_F>0$ denote a decay rate such that
$\dot V_F\leq-\lambda_FV_F$ along the nominal follower
dynamics, as provided by the weighted Poincaré inequality;
(iii) $\sup_{t\geq0}\|\rho_F(t,\cdot)\|_{W^{1,\infty}}
\leq C_\rho$, and
$\mathbf g,\nabla\cdot\mathbf g\in\mathcal L^1(\Omega)$.

If $\limsup_{t\to\infty}\|\nabla\cdot(\rho_L^\star\delta\mathbf u)\|_2\leq\eta_u$, then
\begin{subequations}
\begin{align}
    \limsup_{t\to\infty}\|e_L\|_2
    &\leq\frac{2\eta_u}{\lambda_L},
    \label{eq: tracking_error_leader}\\
    \limsup_{t\to\infty}
    \left\|\frac{e_F}{\sqrt{\bar\rho_F}}\right\|_2
    &\leq\frac{2\tilde K_\rho}{\lambda_F}
    \limsup_{t\to\infty}\|e_L\|_2
    \leq\frac{4\tilde K_\rho}{\lambda_L\lambda_F}\eta_u.
    \label{eq: tracking_error_follower}
\end{align}
\end{subequations}
where
$\tilde K_\rho:=
C_\rho(\|\mathbf g\|_1+\|\nabla\cdot\mathbf g\|_1)/\sqrt{\underline{q}_F}$.
In particular, if $\|\nabla\cdot(\rho_L^\star\delta\mathbf u)\|_2\to0$, both densities converge to their targets in $\mathcal L^2(\Omega)$.
\end{prop}

\begin{pf}
\emph{Leader.}
Since $\mathbf u^\star=\mathbf u+\delta\mathbf u$,
\[
    \partial_t e_L
    =\nabla\cdot(\rho_L^\star\mathbf u)
    +\nabla\cdot(\rho_L^\star\delta\mathbf u).
\]
By (i) and the Cauchy-Schwarz inequality,
\[
    \dot V_L
    \leq-\lambda_LV_L+\sqrt{2V_L}\,
    \|\nabla\cdot(\rho_L^\star\delta\mathbf u)\|_2.
\]
Applying the comparison lemma to $W_L:=\sqrt{V_L}$, with a standard regularization at $W_L=0$, proves~\eqref{eq: tracking_error_leader}.

\emph{Follower.}
Let $\delta\mathbf v_F:=
\mathbf g*\rho_L^\star-\bar{\mathbf v}_F=-\mathbf g*e_L$.
Using (ii), the error dynamics becomes
\[
    \partial_t e_F
    =-\nabla\cdot(e_F\bar{\mathbf v}_F)
    +D\Delta e_F+\nabla\cdot(\rho_F\delta\mathbf v_F).
\]
For the nominal dynamics, stationarity and periodic integration by parts give
\[
    \dot V_F
    =-D\int_\Omega\bar\rho_F
    \left\|\nabla\left(\frac{e_F}{\bar\rho_F}\right)\right\|^2
    \,\mathrm d\mathbf x
    \leq-\lambda_FV_F,
\]
where the equality follows from periodicity and the divergence-free stationary flux $\bar\rho_F\bar{\mathbf v}_F-D\nabla\bar\rho_F$, whereas the inequality follows from weighted Poincaré and $\int_\Omega e_F\,\mathrm d\mathbf x=0$.

Along the perturbed dynamics, weighted Cauchy-Schwarz therefore yields
\[
    \dot V_F
    \leq-\lambda_FV_F+
    \sqrt{2V_F}\,
    \left\|
    \frac{\nabla\cdot(\rho_F\delta\mathbf v_F)}
    {\sqrt{\bar\rho_F}}
    \right\|_2.
\]
By (iii), Young's convolution inequality $\|f*g\|_2\le\|f\|_1\|g\|_2$~\citep{bogachev2007measure}, and $\bar\rho_F\geq \underline{q}_F$,
\[
\begin{aligned}
    \left\|
    \frac{\nabla\cdot(\rho_F\delta\mathbf v_F)}
    {\sqrt{\bar\rho_F}}
    \right\|_2
    &\leq\frac{C_\rho}{\sqrt{\underline{q}_F}}
    \bigl(\|\delta\mathbf v_F\|_2
    +\|\nabla\cdot\delta\mathbf v_F\|_2\bigr)\\
    &\leq\tilde K_\rho\|e_L\|_2.
\end{aligned}
\]
Applying the same comparison argument to $W_F:=\sqrt{V_F}$ and using~\eqref{eq: tracking_error_leader} proves~\eqref{eq: tracking_error_follower}. If the perturbation vanishes, both errors converge to zero; follower convergence in $\mathcal L^2(\Omega)$ follows from $\|e_F\|_2\leq\sqrt{\overline q_F} \|e_F/\sqrt{\bar\rho_F}\|_2$.
\qed
\end{pf}

\section{Numerical simulations}
\label{sec:numerical_simulations}
We consider two scenarios, both implemented through a safety filter acting on the leader control: one in which only follower safety is enforced, and a more restrictive one in which both leader and follower safety are enforced. In both, the desired follower density is the von Mises profile
\[
    \bar{\rho}_F(\mathbf{x})
    \propto \exp\big(
        k_1\cos(x_1)
        +k_2\cos(x_2)
    \big),
\]
normalized so that $\int_\Omega\bar\rho_F\,\mathrm d\mathbf x=M_F$, with $k_1=k_2=9$ chosen to confine a fraction $0.999$ of followers within a goal region of radius $1$ centered at the origin. We consider $N_L=100$ leaders and $N_F=720$ followers, with $D=0.001$, in $\Omega=[-\pi,\pi]^2$.

The dangerous regions are the union of five disks of radius $0.5$, whose centers are at least $0.75$ from the boundary of the goal region, with a minimum distance of $0.25$ between the boundaries of any two disks.
The obstacle density $\rho_O$ is the indicator function of this union, normalized to unit mass. Both populations are uniformly initialized in $\Omega$, excluding a ball of radius 0.75 around each disk center. The safety functionals~\eqref{eq:barrier_functional} use $\varepsilon_F=\varepsilon_L=0.001$.

The interaction kernel parameters are $L_r=\pi$ and $\eta=10^{-8}$, the Gaussian safety kernel width is $\sigma=0.15$, and the gains of the nominal controller are taken from~\citep{maffettone2025leader-follower}.
Both periodic kernels and the required derivatives are evaluated using symmetrically truncated sums over periodic images. We choose $\alpha_{F,0}(t,s)=\gamma_0(t)s$,
$\alpha_{F,1}(s)=\gamma_1s$, and
$\alpha_L(s)=\gamma_Ls$, with
$\gamma_0(0)=\gamma_1=\gamma_L=0.1$.

The density dynamics are approximated by particles,
whose positions are propagated through the
Euler-Maruyama scheme with time step $\Delta t=0.01$
up to the final time $T=200$. Each density is represented by the empirical measure
\[
    \hat{\rho}_i(t,\mathbf{x}) = \frac{1}{N_L+N_F}\sum_{j=1}^{N_i}\delta\!\left(\mathbf{x}-\mathbf{x}_j(t)\right),
\]
where $\delta$ denotes the periodic Dirac delta on $\Omega$, so that $\int_\Omega \hat{\rho}_i\,\mathrm d\mathbf{x}=M_i$ for $i \in \{L,F \}$. The safety constraints are imposed in sampled-data form: at each sampling instant $t_k=k\Delta t$, the terms of the HO-MF-CBF conditions are evaluated from the current empirical densities, the QP~\eqref{eq:qp_leader_follower} is solved at the mean-field level, and the resulting velocity field $\mathbf u^\star$ is then evaluated at the leader positions and held over the sampling interval. The initial conditions~\eqref{eq:initial_conditions} and~\eqref{eq:initial_condition_leader} are satisfied in all simulations. 

In all simulations, the QP remained feasible, consistently with Proposition~\ref{prop:exact_feasibility}; hence, the adaptive
law of Proposition~\ref{prop:adaptive_feasibility} yielded $\nu_0^\star=0$ throughout, leaving $\gamma_0$ unchanged.
For each scenario, we run 50 independent simulations and monitor the kernel-weighted overlaps $R_k(\rho_i,\rho_O)$ and the fraction $\chi$ of followers within the goal region.

\emph{Follower safety.}
We first enforce only follower safety and compare with 50 simulations without any CBF. The second-order filter keeps $R_k(\rho_F,\rho_O)$ below $\varepsilon_F$ throughout the evolution, whereas the unconstrained controller exceeds the threshold during the transient (Fig.~\ref{fig:temporal_evolution}). To achieve this, the leaders approach the dangerous regions to generate the repulsive field that pushes the followers away, resulting in a larger $R_k(\rho_L,\rho_O)$ than in the unconstrained case during the transient. After the transient, the leader overlaps become nearly identical; since leader safety is not enforced, both remain above $\varepsilon_L$ for almost the entire simulation (Fig.~\ref{fig:total}). The filter slightly delays goal reaching (Fig.~\ref{fig: chi}), while the final value of $\langle\chi(T)\rangle$ is equal to $1$ both with and without CBF. In the representative example of Figs.~\ref{fig:start}-\ref{fig:end}, all followers reach the goal region.

\emph{Leader and follower safety.}
We then enforce both constraints and compare the proposed filter, CBF(2), with a first-order filter, CBF(1), enforcing leader safety only, to assess whether leader safety alone indirectly preserves follower safety. It does not: under CBF(1), $R_k(\rho_F,\rho_O)$ exceeds $\varepsilon_F$ during the transient, whereas under CBF(2) it remains below the threshold throughout (Fig.~\ref{fig:temporal_evolution2}), demonstrating the role of the higher-order follower constraint in ensuring safety under indirect actuation. Under both filters, $R_k(\rho_L,\rho_O)$ rapidly approaches $\varepsilon_L$ and then remains close to it, with small numerical oscillations around the boundary (Fig.~\ref{fig:total2}). Thus, comparable leader overlap is accompanied by different follower safety outcomes. The final value of $\chi$ is equal to $1$ under both filters (Fig.~\ref{fig:chi2}), although CBF(1) reaches $\chi\geq0.99$ earlier than CBF(2). In the representative example of Figs.~\ref{fig:start2}-\ref{fig:end2}, all followers reach the goal region, while the leader overlap stays within $\varepsilon_L$, in contrast with Fig.~\ref{fig:end}.

\section{Conclusion}
\label{sec:conclusions}
In this work, we developed a Higher-Order Mean-Field Control Barrier Function framework for interconnected density systems with indirect actuation. We characterized the relative degree of a safety functional in terms of the directed interaction graph among populations: if the constrained population lies at graph distance $d$ from the directly actuated one, the control input enters the $(d+1)$-th time derivative of the safety functional, provided that the corresponding control coefficient is nonzero. Importantly, the control-dependent part of this derivative remains affine in the control velocity field. As a result, the associated higher-order mean-field safety condition can still be enforced through a convex quadratic program.

We applied our framework to a leader-follower density control problem, where the follower population is not directly actuated and its safety must be enforced through the controlled leader density. In this setting, we showed that the follower safety functional has relative degree two with respect to the leader control field, whereas the leader safety functional has relative degree one. This naturally leads to a safety filter defined by affine constraints in the leader control input.

We also discussed feasibility issues that arise when safety constraints are imposed simultaneously on leader and follower populations. In particular, we showed that, without control bounds and apart from the degenerate case of a vanishing follower control coefficient, the two barrier conditions can conflict only when their constraint directions are antiparallel. Numerical simulations in cluttered environments illustrated that the proposed HO-MF-CBF filter can preserve follower safety during the transient phase, while highlighting the safety-performance trade-off.

Future work will address feasibility under bounded control authority, robust safety in the presence of uncertain agent dynamics or disturbances, and the derivation of finite-agent safety guarantees from mean-field barrier conditions.

\section*{Declaration of generative AI and AI-assisted technologies in the writing process}
During the preparation of this work the authors used a generative AI assistant in order to improve the grammar and readability of the manuscript. After using this tool, the authors reviewed and edited the content as needed and take full responsibility for the content of the published article.

\bibliographystyle{plainnat}
\bibliography{autosam}

\end{document}